\documentclass[12pt]{amsart} 
\usepackage{amsmath,amssymb,amsthm,url,hyperref,dsfont}
\usepackage{graphicx}
\usepackage[english]{babel}

\hypersetup{
    colorlinks,
    citecolor=black,
    filecolor=black,
    linkcolor=black,
    urlcolor=black
}

\numberwithin{equation}{section} 

\usepackage[T1]{fontenc}

\newtheorem{theorem}{Theorem}[section]

\newtheorem{proposition}[theorem]{Proposition}
\newtheorem{lemma}[theorem]{Lemma}

\def\irr#1{{\rm Irr}(#1)}
\def\irrr#1#2 {\irr {#1 \mid #2}}
\providecommand{\abs}[1]{\lvert#1\rvert}
\providecommand{\norm}[1]{\left\lVert#1\right\rVert}
\newcommand{\R}{\mathbb R}
\newcommand{\1}{\mathbf 1}

\newcommand{\E}{\mathbb E}

\allowdisplaybreaks

\begin{document}

\title[Sharp bounds for marginal densities]{On sharp bounds for
  marginal densities of product measures} 

\author[G. Livshyts, G. Paouris, P. Pivovarov]{Galyna Livshyts,
  Grigoris Paouris, Peter Pivovarov}

\begin{abstract}
  We discuss optimal constants in a recent result of Rudelson and
  Vershynin on marginal densities.  We show that if $f$ is a
  probability density on $\R^n$ of the form $f(x)=\prod_{i=1}^n
  f_i(x_i)$, where each $f_i$ is a density on $\R$, say bounded by
  one, then the density of any marginal $\pi_E(f)$ is bounded by
  $2^{k/2}$, where $k$ is the dimension of $E$. The proof relies on an
  adaptation of Ball's approach to cube slicing, carried out for
  functions.  Motivated by inequalities for dual affine
  quermassintegrals, we also prove an isoperimetric inequality for
  certain averages of the marginals of such $f$ for which the cube is
  the extremal case.
\end{abstract}
\maketitle

\section{Introduction}
In this note we present an alternate approach to a recent theorem of
Rudelson and Vershynin on marginal densities of product measures
\cite{RV_IMRN}. To fix the notation, if $f$ is a probability density
on Euclidean space $\R^n$ and $E$ is a subspace, the marginal density
of $f$ on $E$ is defined by
$$ \pi_E(f)(x) = \int_{E^{\perp}+x}f(y)dy \quad (x\in E).
$$ In \cite{RV_IMRN}, it is proved that if $f(x) = \prod_{i=1}^n f_i(x_i)$,
where each $f_i$ is a density on $\R$, bounded by $1$, then for any
$k\in \{1,\ldots,n-1\}$, and any subspace $E$ of dimension $k$,
\begin{equation}
  \label{eqn:RV}
  \norm{\pi_E(f)}^{1/k}_{L^{\infty}(E)}\leq C,
\end{equation}
where $C$ is an absolute constant.  

In \cite{RV_IMRN}, it is pointed out that when $k=1$, the constant $C$
in \eqref{eqn:RV} may be taken to be $\sqrt{2}$.  This follows from a
theorem of Rogozin \cite{Rogozin_TVP}, which reduces the problem to
$f=\mathds{1}_{Q_n}$ where $Q_n=[-1/2,1/2]^n$ is the unit cube,
together with Ball's theorem \cite{Ball_PAMS}, \cite{Ball_GAFA} on
slices of $Q_n$.  More precisely, one can formulate Rogozin's Theorem
as follows: if $\theta$ is a unit vector with linear span $[\theta]$,
then
\begin{equation}
  \label{eqn:Rogozin}
  \norm{\pi_{[\theta]}(f)}_{L^{\infty}([\theta])} \leq
  \norm{\pi_{[\theta]}(\mathds{1}_{Q_n})}_{L^{\infty}([\theta])}
\end{equation}for any $f$ in the class
\begin{equation*}
  \mathcal{F}_n=\left\{f(x) =\prod_{i=1}^nf_i(x_i):
  \norm{f_i}_{L^{\infty}(\R)}\leq 1 = \norm{f_i}_{L^{1}(\R)},\;
  i=1,\ldots,n \right\}.
\end{equation*}
By definition of the marginal density and the Brunn-Minkowski
inequality, 
\begin{eqnarray*}
  \norm{\pi_{[\theta]}(\mathds{1}_{Q_n})}_{L^{\infty}([\theta])}& =&
  \max_{x\in [\theta]} \abs{Q_n\cap ({\theta}^{\perp}+x)}_{n-1} \\ &
  = &\abs{Q_n\cap{\theta}^{\perp}}_{n-1},
\end{eqnarray*} 
where $\abs{\cdot}_{n-1}$ denotes $(n-1)$-dimensional Lebesgue
measure.  Ball's theorem gives $\abs{Q_n\cap \theta^{\perp}}_{n-1}\leq
\sqrt{2}$, which shows $C=\sqrt{2}$ works in \eqref{eqn:RV}.

Since Ball's theorem holds in higher dimensions, i.e.,
\begin{equation}
\label{eqn:Ball_intro}
\max_{E\in G_{n,k}}\abs{Q_n\cap E^{\perp}}_{n-k}^{1/k}\leq \sqrt{2}\quad (k\geq 1),
\end{equation}
where $G_{n,k}$ is the Grassmannian of all $k$-dimensional subspaces
of $\R^n$, it is natural to expect that $C=\sqrt{2}$ works in
\eqref{eqn:RV} for all $k>1$. However, in the absence of a
multi-dimensional analogue of Rogozin's result (\ref{eqn:Rogozin}),
the authors of \cite{RV_IMRN} prove \eqref{eqn:RV} with an absolute
constant $C$ via different means.

Our goal is to show that one can determine the optimal $C$ for
suitable $k>1$ directly by adapting Ball's arguments giving
\eqref{eqn:Ball_intro}, and a related estimate, to the functional
setting. The main result of this paper is the following theorem.

\begin{theorem}
  \label{main}
    Let $1\leq k< n$ and $E\in G_{n,k}$.  Then there exists a
    collection of numbers $\{\gamma_i\}_{i=1}^n\subset [0,1]$ with
    $\sum_{i=1}^n \gamma_i=k$ such that for any bounded functions
    $f_1,\ldots,f_n:\R\rightarrow [0,\infty)$ with
      $\norm{f_i}_{L^1(\R)}=1$ for $i=1,\ldots,n$, the product
      $f(x)=\prod_{i=1}^n f_i(x_i)$ satisfies
    \begin{equation}
      \norm{\pi_E(f)}_{L^{\infty}(E)}\leq
      \min\left(\left(\frac{n}{n-k}\right)^{\frac{n-k}{2}},2^{k/2}\right)
       \prod_{i=1}^n \norm{f_i}^{\gamma_i}_{L^{\infty}(\R)}.
      \end{equation}
\end{theorem}

In particular, the theorem implies that if $f\in \mathcal{F}_n$ and
$E\in G_{n,k}$, then
  \begin{equation}
    \label{eqn:consequence}
    \norm{\pi_E(f)}_{L^{\infty}(E)}\leq
    \min\left(\left(\frac{n}{n-k}\right)^{\frac{n-k}{2}},2^{k/2}\right).
  \end{equation}

As noted in \cite{Ball_GAFA}, if $f=\mathds{1}_{Q_n}$, the bound
$\left(\frac{n}{n-k}\right)^{(n-k)/2}$ is achieved when $n-k$ divides
$n$ and $E_0\in G_{n,k}$ is chosen so that $Q_n\cap E_0^{\perp}$ is a
cube of suitable volume; note that
$\left(\frac{n}{n-k}\right)^{\frac{n-k}{2}}\leq e^{k/2}$. When $k\leq
n/2$, the bound $2^{k/2}$ is sharp when $Q_n\cap E_0^{\perp}$ is a box
of suitable volume. Thus for such $k$, Theorem \ref{main} implies
\begin{equation}
  \label{eqn:compare_cube}
  \sup_{E\in G_{n,k}} \norm{\pi_E(f)}_{L^{\infty}(E)} \leq \sup_{E\in
    G_{n,k}}\norm{\pi_E(\mathds{1}_{Q_n})}_{L^{\infty}(E)} \quad (f\in
  \mathcal{F}_n).
\end{equation}
In terms of random vectors, if $X\in \R^n$ is distributed according to
$f$, then the density of the orthogonal projection $P_EX$ of $X$ onto
$E$ is simply $\pi_E(f)$. Thus if $X$ has density $f\in \mathcal{F}_n$
and $Y$ has density $\mathds{1}_{Q_n}$, the density of $P_EX$ is
uniformly bounded above by the value of the density of $P_{E_0}Y$ at
the origin (with $E_0$ chosen as above). 

For another probabilistic consequence, note that \eqref{eqn:consequence}
implies the following small-ball probability: for each $z\in E$,
\begin{equation}
  \mathbb{P}\left(\norm{P_E X -z}\leq \varepsilon \sqrt{k}\right)\leq
  (C\sqrt{2e\pi}\varepsilon)^k \quad (\varepsilon >0),
\end{equation}
where $\norm{\cdot}$ denotes the Euclidean norm; this was part of the
motivation in \cite{RV_IMRN}; see also Tikhomirov \cite{Tikhomirov}
for recent work on such estimates and their use in random matrix
theory.


Ball's approach to cube slicing \cite{Ball_PAMS}, \cite{Ball_GAFA} has
been adapted to a variety of related problems. We mention just a
sample and refer the reader to the references therein; see Barthe's
multidimensional version \cite{Barthe_IM} of the Brascamp-Lieb
inequality \cite{BL} and its normalized form; the use of the latter by
Gluskin \cite{Gluskin_GAFA} for slices of products of measurable sets;
Koldobsky and K\"{o}nig \cite{KK_measure} for problems involving
measures other than volume; Brzezinski \cite{Brzezinski} for recent
work on slices of products of Euclidean balls; Bobkov and Chistyakov
\cite{BC_12}, \cite{BC_15} for connections to sums of independent
random variables.

Recently, bounds for marginals of arbitrary bounded densities have
been found by S. Dann and the second and third-named authors
\cite{DPP}. They obtain extremal inequalities for certain averages,
e.g., for any $k$ and $f:\R^n\rightarrow \R^+$ satisfying
$\norm{f}_{L^{\infty}(\R^n)}\leq 1 = \norm{f}_{L^1(\R^n)}$ and
$f(0)=\norm{f}_{L^{\infty}(\R^n)}$, one has
\begin{equation}
  \label{eqn:DPP}
  \int_{G_{n,k}}\pi_E(f)(0)^n d\mu_{n,k}(E) \leq
  \int_{G_{n,k}}\pi_E(\mathds{1}_{D_n})(0)^nd\mu_{n,k}(E),
\end{equation}
where $\mu_{n,k}$ is the Haar probability measure on $G_{n,k}$ and
$D_n$ is the Euclidean ball in $\R^n$ of volume one centered at the
origin.

Using an idea from the proof of Theorem \ref{main}, we also obtain the
following strengthening of \eqref{eqn:DPP} within the class
$\mathcal{F}_n$.

\begin{proposition}
  \label{prop:avg}
  Let $1\leq k < n$ and $f\in \mathcal{F}_n$. Then
  \begin{equation}
    \label{eqn:avg}
    \int_{G_{n,k}}\pi_E(f)(0)^n d\mu_{n,k}(E) \leq
    \int_{G_{n,k}}\pi_E(\mathds{1}_{Q_n})(0)^n d\mu_{n,k}(E).
  \end{equation}
\end{proposition}

The latter can be seen as a type of ``average'' domination of
marginals of $\mathds{1}_{Q_n}$ over those of $f\in
\mathcal{F}_n$. This complements the pointwise domination of Rogozin's
Theorem \eqref{eqn:Rogozin} when $k=1$ and the worst-case comparison
in \eqref{eqn:compare_cube} when $k\leq n/2$ or $n-k$ divides $n$. As
in \cite{DPP}, inequality \eqref{eqn:avg} is another step in extending
results about dual affine quermassintegrals (we recall the definition
in \S \ref{section:proofs}) from convex sets to functions in order to
quantify characteristics of high-dimensional probability measures.

\section{Preliminaries}

The setting is Euclidean space $\R^n$ with the standard basis
$\{e_1,\ldots,e_n\}$, usual inner product $\langle\cdot,\cdot\rangle$,
Euclidean norm $\norm{\cdot}$ and unit sphere $S^{n-1}$. We reserve
$\abs{\cdot}_k$ for $k$-dimensional Lebesgue measure; the subscript
$k$ will be omitted if the context is clear. We denote the positive
reals by $\R^{+}$.

For a Borel set $A\subset \R$ of finite Lebesgue measure, the
symmetric rearrangement $A^*$ of $A$ is the symmetric interval
$A^*=[-\abs{A}/2,\abs{A}/2]$. For an integrable function
$g:\R\rightarrow [0,\infty)$ its symmetric decreasing rearrangement
  $g^*$ is defined by
\begin{equation*}
g^*(x)=\int_0^{\infty} \1_{\{g>t\}^*}(x) dt.
\end{equation*}
This can be compared with the layer-cake representation of $g$:
\begin{equation}
  \label{eqn:layer}
  g(x)=\int_0^{\infty} \1_{\{g>t\}}(x) dt
  =\int_0^{\norm{g}_{L^{\infty}(\R)}}\1_{\{g>t\}}(x) dt.
\end{equation}
Then $g$ and $g^*$ are equimeasurable, i.e., $\abs{\{g>t\}}
=\abs{\{g^*>t\}}$ for each $t>0$. In particular, $\norm{g}_{L^p(\R)} =
\norm{g^*}_{L^p(\R)}$ for $1\leq p\leq \infty$.

\section{Adapting Ball's arguments}

We start with the following basic fact used in \cite{Ball_PAMS},
\cite{Ball_GAFA}, proved for completeness.

\begin{lemma}
  \label{lemma:project_like_ball}
  Let $b=(b_1,\ldots,b_n)\in S^{n-1}$ and let $A$ be a measurable
  subset of $b^{\perp}$ with $\mathop{\rm dim}(\mathop{\rm
    span}(A))=k\in \{1,\ldots,n-1\}$.  Then for each $1\leq i\leq n$,
  \begin{equation*}
  \abs{b_i}\abs{A}_k\leq \abs{P_i(A)}_k,
  \end{equation*}where
  $P_i=P_{e_i^{\perp}}$ is the orthogonal projection onto
  $e_i^{\perp}$.
\end{lemma}

\begin{proof}
  We may assume that $P_i:b^{\perp}\rightarrow e_i^{\perp}$ is injective
  (otherwise $b_i=0$ and the inequality is trivial). We may also assume
  that $b\not = \pm e_i$ (otherwise equality holds). Let
  $v_1,\ldots,v_k$ be an orthonormal basis of $\mathop{\rm span}(A)$
  with $ v_1 =\frac{1}{\norm{P_ib}}e_i - \frac{b_i}{\norm{P_ib}} b $ and
  $v_2,\ldots,v_k$ orthogonal to both $e_i$ and $b$. Then $P_i(v_i)=v_i$
  for $i\geq 2$, and ${\norm{P_i (v_1)}}=|b_i|.$ Consider the
  $k$-dimensional cube $C=\prod_{i=1}^{k}[0,v_i]\subset \mathop{\rm
    span}(A)$. Then $P_i C$ is a $k$-dimensional box in $e_i^{\perp}$
  with the sides $\abs{b_i},1,...,1$. Hence $|P_i(C)|_{k}=
  \abs{b_i}\abs{C}_{k}$.  Thus the lemma is true for coordinate cubes in
  $\mathop{\rm span}(A)$.  The inequality follows by approximating $A$
  by disjoint cubes. Since $P_i|_{b^{\perp}}$ is injective, the images
  of such cubes under $P_i$ remain disjoint.
\end{proof}

The first ingredient in Ball's approach is the following integral
inequality \cite{Ball_PAMS}.

\begin{theorem}
  \label{thm:Ball1}
  For every $p\geq 2,$
  \begin{equation}
    \frac{1}{\pi}\int_{-\infty}^{\infty} \left|\frac{\sin t}{t}\right|^p dt
    \leq \sqrt{\frac{2}{p}}.
  \end{equation}
\end{theorem}

The second ingredient is Ball's normalized form \cite{Ball_GAFA} of
the Brascamp-Lieb inequality \cite{BL}.

\begin{theorem}
  \label{thm:Ball2}
  Let $u_1,...,u_n$ be unit vectors in $\R^n,$ $m\geq n,$ and
  $c_1,...,c_m>0$ satisfying $\sum_1^m c_i u_i\otimes u_i=I_n.$ Then for
  integrable functions $f_1,...,f_m:\R\rightarrow [0,\infty),$
    \begin{equation}
    \int_{\R^n} \prod_{i=1}^m f_i(\langle u_i,x\rangle)^{c_i}dx\leq
    \prod_{i=1}^m \left(\int_{\R} f_i\right)^{c_i}.
    \end{equation}
    There is equality if
    the $f_i'$s are identical Gaussian densities.
\end{theorem}

We will also use the following standard fact, proved for the
convenience of the reader.

\begin{lemma}
  \label{lemma:marginal_id}
  Let $1\leq k<n$ and $E\in G_{n,k}$. Then there exist vectors
  $w_1,\ldots,w_n$ in $\R^{n-k}=\mathop{\rm
    span}\{e_1,\ldots,e_{n-k}\}$ such that $I_{n-k} = \sum_{i=1}^{n}
  w_{i}\otimes w_{i}$ and for any integrable function
  $f(x)=\prod_{i=1}^n f_i(x_i)$ with $f_i:\R\rightarrow [0,\infty)$,
  \begin{equation}
    \label{eqn:marginal_id}
    \pi_E(f)(0) = \int_{\mathbb R^{n-k}} \prod_{i=1}^{n} f_{i}
    (\langle y, w_{i}\rangle ) dy.
\end{equation}

\end{lemma}

\begin{proof}
Let $v_{1}, \ldots , v_{n-k}\in \R^n$ be an orthonormal basis of
$E^{\perp}$ and let $w_i$ be defined by
  \begin{equation*} 
    w_{i}:= (\langle v_{1} , e_{i}\rangle , \ldots , \langle v_{n-k},
    e_{i}\rangle ),\quad 1\leq i \leq n.
  \end{equation*}In matrix terms, if $V$ is the $n\times (n-k)$ matrix
  with columns $v_1,\ldots,v_{n-k}$, then $w_i= V^{T}e_i$, where $V^T$
  is the transpose of $V$.  Then
  \begin{equation*}
    \sum_{i=1}^n w_iw_i^{T} = \sum_{i=1}^n V^T e_i e_i^T V = I_{\R^{n-k}}
  \end{equation*}
  and 
  \begin{eqnarray*}
    \pi_E(f) (0) & = & \int_{E^{\perp}}f(y) dy  = 
    \int_{\R^{n-k}}f\left(\sum_{i=1}^{n-k} y_i v_i\right) dy \\ & = &
    \int_{\R^{n-k}} \prod_{i=1}^n f_i (\langle V y,e_i\rangle) dy  =
    \int_{\R^{n-k}} \prod_{i=1}^n f_i(\langle y,w_i \rangle) dy.
  \end{eqnarray*}
\end{proof}

The following two propositions extend Ball's estimates on slices of
the cube to coordinate boxes. It is essential that we obtain estimates
that are uniform among all such boxes. The proofs draw heavily on
\cite{Ball_GAFA}.

\begin{proposition}
  \label{prop:box_1}
  Let $1\leq k < n$ and $H\in G_{n,n-k}$. Then there exists
  $\{\beta_i\}_{i=1}^n\subset [0,1]$ with $\sum_{i=1}^n \beta_i=n-k$
  such that for any $z_1,\ldots,z_n\in \R^{+}$, the box
  $B=\prod_{i=1}^n[-z_i/2,z_i/2]$ satisfies
  \begin{equation}
    \abs{B\cap H}\leq \left(
    \frac{n}{n-k}\right)^{\frac{n-k}{2}}\prod_{i=1}^n z_i^{\beta_i}.
  \end{equation}
\end{proposition}

\begin{proof}
  Let $w_1,\ldots,w_n$ be as in Lemma \ref{lemma:marginal_id} (with
  $k$ and $n-k$ interchanged).  For $i=1,\ldots,n$, let
  $u_i=w_i/\norm{w_i}$ and $a_i = \norm{w_i}$ and $\beta_i=a_i^2$.
  For $z_1,\ldots,z_n\in \R^{+}$, we apply \eqref{eqn:marginal_id}
  with $f=\mathds{1}_B$ and $E=H^{\perp}$ to get
  \begin{eqnarray*}
    \abs{B\cap H} &=& \pi_{H^{\perp}}(\mathds{1}_{B})(0)\\ & = & \int_{\R^k}
     \prod_{i=1}^{n} \1_{[-\frac{z_i}{2},\frac{z_i}{2}]}(\langle y,
     w_{i}\rangle ) dy\\ & = & \int_{\R^k} \prod_{i=1}^{n}
     \1_{[-\frac{z_i}{2a_i},\frac{z_i}{2a_i}]}(\langle y, u_i\rangle)
     dy.
   \end{eqnarray*} 
   Using Theorem \ref{thm:Ball2}, the latter is at most
   \begin{equation}\label{equation:estimate1-2}
     \prod_{i=1}^{n} \left( \int_{\R}
     \1_{[-\frac{z_i}{2a_i},\frac{z_i}{2a_i}]}(t)dt
     \right)^{a_i^2}=\prod_{i=1}^{n} \left( \frac{z_i}{a_i}
     \right)^{a_i^2}.
   \end{equation} 
   As in \cite[Proof of Proposition 4]{Ball_GAFA}, we use the bound
   $$\prod_{i=1}^{n} a_i ^{-a_i^2}\leq
   \left(\frac{n}{n-k}\right)^{\frac{n-k}{2}},$$ from which the lemma
   follows.
 \end{proof}

 \begin{proposition}
   \label{prop:box_2}
   Let $1\leq k \leq n/2$ and $H\in G_{n,n-k}$. Then there exists
   $\{\beta_j\}_{j=1}^n\subset [0,1]$ with $\sum_{i=1}^n \beta_i=n-k$
   such that for any $z_1,\ldots,z_n\in \R^{+}$, the box
   $B=\prod_{j=1}^n[-z_j/2,z_j/2]$ satisfies 
   \begin{equation} |B\cap H|\leq 2^{k/2}\prod_{j=1}^n z_j^{\beta_j}.
   \end{equation}
 \end{proposition}

 \begin{proof}
   Assume first that all unit vectors $b=(b_1,\ldots,b_n)\in
   H^{\perp}$, satisfy $b_i\leq \frac{1}{\sqrt{2}}$ for each
   $i=1,\ldots,n$. Let $\widetilde{P}=P_{H^{\perp}}$ be the orthogonal
   projection onto $H^{\perp}$. For $i=1,\ldots,n$, let
   $u_i=\frac{\widetilde{P}e_i}{\norm{\widetilde{P}e_i}}$ and
   $a_i={||\widetilde{P}e_i||}$. Note that
   \begin{equation}
     \label{eqn:I_a}
     (i) \sum_{i=1}^n a_i^2 u_i\otimes u_i = I_{H^{\perp}}, \quad\quad (ii)
     \sum_{i=1}^n a_i^2 = k.
   \end{equation}
   By our assumption, all $a_i\leq \frac{1}{\sqrt{2}}$, since $a_i$ is
   the $i$-th coordinate of the unit vector $u_i$ in $H^{\perp}$.

   Assume for the time being that $z_1,\ldots,z_n$ are fixed and
   satisfy $\abs{B}=\prod_{j=1}^n z_j=1$. Let $X=(X_1,\ldots,X_n)$ be a random
   vector with density $\mathds{1}_B$ and $Y=(Y_1,\ldots,Y_n)$ be a
   random vector with density $\mathds{1}_{Q_n}$. The characteristic
   function $\Phi: H^{\perp}\rightarrow \R$ of $\widetilde{P}X$
   satisfies
   \begin{eqnarray*}
     \Phi(w)&=&\E \exp\left(i\langle w,\widetilde{P}X\rangle\right) \\& = &\E
     \exp\left(i\sum_{j=1}^n X_j a_j\langle w,u_j\rangle\right) \\ &=& \E
     \exp\left(i\sum_{j=1}^n Y_j z_j a_j\langle w,u_j\rangle\right)\\ & =
     &\prod_{j=1}^n \frac{2\sin\frac{1}{2}z_j a_j \langle
       w,u_j\rangle}{z_j a_j \langle w,u_j\rangle}.
   \end{eqnarray*}
   Since the marginal density $\pi_{H^{\perp}}(\mathds{1}_{B})$ is
   continuous, we can apply the Fourier inversion formula (e.g.,
   \cite[Theorem 9.5.4]{Dudley}) to obtain
   \begin{eqnarray*}
     \label{BLB}
     \abs{B\cap H}& = &\pi_{H^{\perp}}(\mathds{1}_B)(0)\\
     &=&\frac{1}{(2\pi)^k}\int_{H^{\perp}}
     \Phi(w)dw\\ & = &\frac{1}{\pi^k}\int_{H^{\perp}}\prod_{j=1}^n
     \frac{\sin z_j a_j \langle w,u_j\rangle}{z_j a_j \langle
       w,u_j\rangle} dw\\ &\leq
     &\frac{1}{\pi^k}\int_{H^{\perp}}\prod_{j=1}^n \left|\frac{\sin z_j
       a_j \langle w,u_j\rangle}{z_j a_j \langle w,u_j\rangle}
     \right|dw\\ &= &\frac{1}{\pi^k}\int_{H^{\perp}}\prod_{j=1}^n
     \Phi_j(\langle w,u_j\rangle)dw,
   \end{eqnarray*}where $\Phi_j:\R\rightarrow [0,\infty)$ is defined by
     $\Phi_j(t)=\left|\frac{\sin z_j a_j t}{z_j a_j t}\right|.$
     Consequently, Theorem \ref{thm:Ball2} with $c_j=\frac{1}{a_j^2}$
     implies that $\abs{B\cap H}$ is at most
     \begin{equation}
       \label{almost_done}
       \frac{1}{\pi^k}\prod_{j=1}^n \left( \int_{\R}
       \Phi_j(t)^{\frac{1}{a_j^2}} dt\right)^{a_j^2}=\prod_{j=1}^n
       \left( \frac{1}{\pi}\int_{\R} \Phi_j(t)^{\frac{1}{a_j^2}}
       dt\right)^{a_j^2}
     \end{equation}
     (cf. (\ref{eqn:I_a})). Finally, we use Theorem \ref{thm:Ball1}:
     \begin{eqnarray*}
       \frac{1}{\pi}\int_{\R} \Phi_j(t)^{\frac{1}{a_j^2}} dt &=
       &\frac{1}{\pi}\int_{\R} \left|\frac{\sin z_j a_j t}{z_j a_j
         t}\right|^{\frac{1}{a_j^2}} dt\\ &=&\frac{1}{z_j
         a_j}\left(\frac{1}{\pi}\int_{\R} \left|\frac{\sin
         t}{t}\right|^{\frac{1}{a_j^2}} dt\right)\\&\leq &
       \frac{\sqrt{2}}{z_j}.
     \end{eqnarray*}
     In summary, for any $z_1,\ldots,z_n\in \R^{+}$ with
     $\abs{B}=\prod_{j=1}^n z_j=1$, we have
     \begin{equation}
       \label{equation:Ball}
       \abs{B\cap H}\leq
       \prod_{j=1}^n\left(\frac{\sqrt{2}}{z_j}\right)^{a_j^2}=2^{k/2}\prod_{j=1}^n
       z_j^{-a_j^2}.
 \end{equation}
 For an arbitrary box $B=\prod_{j=1}^n[-\frac{z_j}{2},\frac{z_j}{2}]$,
 we get via scaling that
  \begin{equation}\label{equation:Ball-1}
    \abs{B\cap H}\leq 2^{k/2}\prod_{j=1}^n z_j\prod_{j=1}^n
 z_j^{-a_j^2}=2^{k/2}\prod_{j=1}^n z_j^{\beta_j},
 \end{equation}
 where $\beta_j=1-a_j^2.$ Note that we assumed $a_j\leq
 \frac{1}{\sqrt{2}}$, so in fact $\beta_j\in[\frac{1}{2},1]$.

 Suppose now that there exists a unit vector $b\in H^{\perp}$ such
 that $b_i\geq \frac{1}{\sqrt{2}}$ for some $i$. By induction, assume
 the proposition is true for all dimensions at most $n-1$ and for all
 $k$. For $z_1,\ldots,z_n\in \R^{+}$, note that the cylinder
 $$C=\left\{x\in \R^n:\, \norm{x_j}\leq \frac{z_j}{2}\;\;\forall j\neq
 i\right\}$$ satisfies $\abs{B\cap H}\leq \abs{C\cap H}.$ By Lemma
 \ref{lemma:project_like_ball},
 \begin{equation}
   \label{equation:sqrt2}
   \abs{C\cap H}\leq\frac{1}{b_i}\abs{P_i(C\cap
     H)}\leq\sqrt{2}\abs{\widetilde{B}\cap \widetilde{H}},
 \end{equation}
 where $\widetilde{B}$ is an $(n-1)$ dimensional box with sides
 $\{z_j\}_{j\neq i}$ and $\widetilde{H}=P_i H$ is a $(k-1)$-codimensional
 subspace in $\R^{n-1}$. If $k=1$, then
 $$\abs{\widetilde{B}\cap
   \widetilde{H}}=\abs{\widetilde{B}}=\prod_{j\neq i} z_j,$$ and thus
 $$|B\cap H|\leq \sqrt{2}\prod_{j\neq i} z_j,$$ hence the proposition
 holds with $\beta_j=1$ for $j\not =i$. If $k\geq 2,$ we use the
 inductive hypothesis: there exists $\{\beta_j\}_{j\not =i} \subset
 [0,1]$ with $\sum_{j\not =i} \beta_j = n-1-(k-1)=n-k$ such that
 $$\abs{\widetilde{B}\cap \widetilde{H}}\leq 2^{(k-1)/2} \prod_{j\not
   =i} z_j^{\beta_j}.$$ Using (\ref{equation:sqrt2}), and taking
 $\beta_j=0$, we conclude that
 $$\abs{B\cap H}\leq 2^{k/2} \prod_{j \not = i } z_j^{\beta_j}.$$
 \end{proof}

 \section{Proofs of Theorem \ref{main} and Proposition \ref{prop:avg}}

 \label{section:proofs}

 \begin{proof}[Proof of Theorem \ref{main}]  
   By translating if necessary, we may assume that
   \begin{equation}
     \label{eqn:zero}
     \norm{\pi_E(f)}_{L^{\infty}(E)}= \pi_E(f)(0).
   \end{equation}
   For $i=1,...,n$, set $c_i=\norm{f_i}_{L^{\infty}(\R)}$ and consider
   the box $C=\prod_{i=1}^n[0,c_i].$ Let $w_1,\ldots,w_n$ be as in
   Lemma \ref{lemma:marginal_id}.  Using \eqref{eqn:marginal_id}, the
   rearrangement inequality of Rogers \cite{Rogers_JLMS} and
   Brascamp-Lieb-Luttinger \cite{BLL_JFA}, and the layer-cake
   representation \eqref{eqn:layer}, we have
  \begin{eqnarray*}
    \pi_E(f)(0)& = & \int_{\mathbb R^{n-k}} \prod_{i=1}^{n} f_{i}
    (\langle y, w_{i}\rangle ) dy \\ &\leq & \int_{\mathbb R^{n-k}}
    \prod_{i=1}^{n} f^{*}_{i} (\langle y, w_{i}\rangle ) dy\\ & = &
    \int_0^{{c_1}}\cdots \int_0^{{c_n}} \int_{\R^{n-k}}\prod_{i=1}^n
        {\1}_{\{f^*_i > t_i\} }(\langle x, w_i \rangle) dx dt_1\ldots
        dt_n.
  \end{eqnarray*}
  Write $dt=dt_1\ldots dt_n$ and
  $M=\left(\frac{n}{n-k}\right)^{\frac{n-k}{2}}$.  Since each $f^*_i$
  is symmetric and decreasing, the set $\{f^*_i>t_i\}$ is a symmetric
  interval. Consequently, we apply Proposition \ref{prop:box_1} with
  $z_i:=|\{f^*_i>t_i\}|$, $i=1,\ldots,n$, to get
  \begin{eqnarray*}
  \pi_E(f)(0) &\leq & M \int_{{C}} \prod_{i=1}^n
  |\{f^*_i>t_i\}|^{\beta_i} dt \\ & \leq & M { \prod_{i=1}^n
    c_i^{1-\beta_i}\cdot }
  \prod_{i=1}^n||f^*_i||_{L^1(\R)}^{\beta_i}\\ & \leq & M
       {\prod_{i=1}^n c_i^{1-\beta_i}} ;
  \end{eqnarray*}
  here we used Fubini's theorem, H\"{o}lder's inequality and the fact
  that
  \begin{equation}
    \label{eqn:layer*}
    \int_0^{{c_i}} |\{f^*_i>t_i\}|
    dt_i=\norm{f^*_i}_{L^1(\R)}=\norm{f_i}_{L^1(\R)}=1, \quad
    i=1,\ldots,n.
  \end{equation}
  Thus setting $\gamma_i=1-\beta_i$, $i=1,\ldots,n$, we obtain the
  first estimate in the theorem.  Repeating the latter argument with
  $M=2^{k/2}$, using Proposition \ref{prop:box_2}, concludes the proof
  of the theorem.
\end{proof}

Before proving Proposition \ref{prop:avg}, we recall the following
notion, proposed by Lutwak: if $K$ is a convex body in $\R^n$, and
$1\leq k < n$, the dual affine quermassintegrals of $K$ are defined by
\begin{equation}
  \widetilde{\Phi}_{k}(K) =
  \frac{\omega_n}{\omega_k}\left(\int_{G_{n,k}}\abs{K\cap E}^n
  d\mu_{n,k}(E)\right)^{\frac{1}{n}},
\end{equation}
where $\omega_n$ is the volume of the Euclidean ball in $\R^n$ of
radius one; see \cite{Lutwak_PAMS}, \cite{Lutwak_AM} for further
background. Grinberg \cite{Grinberg} proved that
\begin{equation}
  \label{eqn:Grinberg}
  \widetilde{\Phi}_{k}(K)=\widetilde{\Phi}_{k}(SK)
\end{equation}
for each volume-preserving linear transformation $S$; see \cite{DPP}
for a generalization of the latter invariance property for functions. 

\begin{proof}[Proof of Proposition \ref{prop:avg}]
  Let $w_1,\ldots,w_n$ be as in Lemma \ref{lemma:marginal_id}. As in
  the proof of Theorem \ref{main},
\begin{eqnarray*}
 \pi_E(f)(0)& \leq & \int_0^1\cdots \int_0^1
 \int_{\R^{n-k}}\prod_{i=1}^n {\bf 1}_{\{f^*_i > t_i\} }(\langle x,
 w_i \rangle) dx dt_1\ldots dt_n \\ &= & \int_{[0,1]^n} \abs{B(t)\cap E^{\perp}}
 dt,
\end{eqnarray*}
where $B(t)$ is the origin-symmetric box with side-lengths
$\abs{\{f^*_i\geq t_i\}}$, $i=1,\ldots,n$.  Thus
\begin{eqnarray*}
  \lefteqn{\int_{G_{n,k}} \pi_E(f)(0)^n
 d\mu_{n,k}(E)}\\& & \leq \int_{G_{n,k}}\left(\int_{[0,1]^n} \abs{B(t)\cap E^{\perp}}
 dt\right)^n d\mu_{n,k}(E)\\
  &  & = \int_{G_{n,k}}\left(\int_{[0,1]^n} \left(\prod_{j=1}^n
  \abs{\{f^*_j>t_j\}}\right)^{\frac{n-k}{n}} \abs{\widetilde{B}(t)\cap E^{\perp}}
  dt\right)^n d\mu_{n,k}(E),
\end{eqnarray*} where $\widetilde{B}(t) = B(t)/\abs{B(t)}^{1/n}$.
Using H\"{o}lder's inequality (twice), along with \eqref{eqn:layer*}
and \eqref{eqn:Grinberg}, we get
\begin{eqnarray*}
\int_{G_{n,k}} \pi_E(f)(0)^n d\mu_{n,k}(E) &\leq&
\int_{[0,1]^n}\int_{G_{n,k}} \abs{\widetilde{B}(t)\cap E^{\perp}}^n
d\mu_{n,k}(E) dt\\ &=& \int_{G_{n,k}} \abs{Q_n\cap E^{\perp}}^n
d\mu_{n,k}(E).
\end{eqnarray*}

\end{proof}

\section{Acknowledgements}

It is our pleasure to thank Alex Koldobsky for helpful discussions.
Galyna Livshyts would like to thank the University of Missouri and
Texas A\&M University for their hospitality during visits in which
this project started. Galyna Livshyts is supported by a Kent State
University Fellowship and US NSF grant 1101636 (to A. Zvavitch and
D. Ryabogin). Grigoris Paouris is supported by the A. Sloan
Foundation, US NSF grant CAREER-1151711 and BSF grant 2010288. This
work was also partially supported by a grant from the Simons
Foundation (\#317733 to Peter Pivovarov).  \small
\bibliographystyle{amsplain} \bibliography{LPPbib}

\medskip

\noindent
Galyna Livshyts: School of Mathematics, Georgia Institute of
Technology, Atlanta GA 30319, {\tt glivshyts6@math.gatech.edu}
\medskip

\noindent 
Grigoris Paouris: Department of Mathematics (Mailstop 3368), Texas
A\&M University, College Station, TX 77843-3368, {\tt grigoris@math.tamu.edu}
\medskip

\noindent 
Peter Pivovarov: Mathematics Department, University of Missouri,
Columbia, MO 65211, {\tt pivovarovp@missouri.edu}

\end{document}